\documentclass[12pt]{amsart}
\usepackage[colorlinks=true,urlcolor=blue, citecolor=red,linkcolor=blue,linktocpage,pdfpagelabels, bookmarksnumbered,bookmarksopen]{hyperref}
\usepackage[hyperpageref]{backref}
\usepackage{cleveref}
\usepackage{amsthm} 
\usepackage{latexsym,amsmath,amssymb}
\usepackage{accents}
\usepackage{a4wide}
\usepackage{soul}
\usepackage{mathtools} 
\usepackage{xparse} 
\usepackage{tikz, tikz-cd} 
\usepackage{cancel}
\usetikzlibrary{decorations.markings} 
  
\title[Estimates for rational homotopy groups]{Quantitative estimates for fractional Sobolev mappings in rational homotopy groups}
\author{Woongbae Park}
\address[Woongbae Park]{Department of Mathematics,
University of Pittsburgh,
301 Thackeray Hall,
Pittsburgh, PA 15260, USA}
\email{wop5@pitt.edu}

\author{Armin Schikorra}
\address[Armin Schikorra]{Department of Mathematics,
University of Pittsburgh,
301 Thackeray Hall,
Pittsburgh, PA 15260, USA}
\email{armin@pitt.edu}

plus 6pt minus 12pt
\def\eps{\varepsilon}

\def\B{{\mathbb{B}}}

\def\N{{\mathbb N}}
\def\n{{\mathcal N}}

\def\S{{\mathbb S}}

\newtheorem{theorem}{Theorem}
\newtheorem{lemma}[theorem]{Lemma}

\newtheorem{proposition}[theorem]{Proposition}

\newtheorem{example}[theorem]{Example}
\newtheorem{question}[theorem]{Question}

\def\dist{{\rm dist\,}}

\def\lip{{\rm Lip\,}}
\def\rank{{\rm rank\,}}

\newcommand{\R}{\mathbb{R}}

\newcommand{\Z}{\mathbb{Z}}

\newcommand{\brac}[1]{\left (#1 \right )}
\newcommand{\abs}[1]{\left |#1 \right |}
\newcommand{\Ep}{\bigwedge\nolimits}

\newcommand{\barint}{
\rule[.036in]{.12in}{.009in}\kern-.16in \displaystyle\int }

\newcommand{\barcal}{\text{$ \rule[.036in]{.11in}{.007in}\kern-.128in\int $}}

\newcommand{\bbbc}{\mathbb C}

\def\mvint_#1{\mathchoice
          {\mathop{\vrule width 6pt height 3 pt depth -2.5pt
                  \kern -8pt \intop}\nolimits_{\kern -3pt #1}}%
          {\mathop{\vrule width 5pt height 3 pt depth -2.6pt
                  \kern -6pt \intop}\nolimits_{#1}}%
          {\mathop{\vrule width 5pt height 3 pt depth -2.6pt
                  \kern -6pt \intop}\nolimits_{#1}}%
          {\mathop{\vrule width 5pt height 3 pt depth -2.6pt
                  \kern -6pt \intop}\nolimits_{#1}}}

\numberwithin{theorem}{section} \numberwithin{equation}{section}

\newcommand{\lap}{\Delta }
\newcommand{\aleq}{\precsim}

\newcommand{\aeq}{\approx}

\newcommand{\Ds}[1]{|D|^{#1}}
\newcommand{\laps}[1]{(-\lap)^{\frac{#1}{2}}}

\newcommand{\lapms}[1]{I^{#1}}

\let\latexchi\chi
\makeatletter
\renewcommand\chi{\@ifnextchar_\sub@chi\latexchi}
\newcommand{\sub@chi}[2]{
  \@ifnextchar^{\subsup@chi{#2}}{\latexchi^{}_{#2}}%
}
\newcommand{\subsup@chi}[3]{
  \latexchi_{#1}^{#3}%
}
\makeatother

\newcommand{\m}{\mathcal{M}}

\renewcommand{\d}{{\rm {deg}}}

\begin{document}

\begin{abstract}
Let $\mathcal{N} \subset \mathbb{R}^M$ be a smooth simply connected compact manifold without boundary. A rational homotopy subgroup of $\pi_{N}(\mathcal{N})$ is represented by a homomorphism \[{\rm deg}: \pi_{N}(\mathcal{N}) \to \mathbb{R}.\] For maps $f: \mathbb{S}^N \to \mathcal{N}$ we give a quantitative estimate of its rational homotopy group element ${\rm deg}([f]) \in \mathbb{R}$ in terms of its fractional Sobolev-norm. That is, we show that for all $\beta \in (\beta_0({\rm deg}),1]$,
\[
 |{\rm deg}([f])|\leq C({\rm deg})\, [f]_{W^{\beta,\frac{N}{\beta}}(\mathbb{S}^N)}^{\frac{N+L({\rm deg})}{\beta}}.
\]
Here $C({\rm deg}) > 0$, $L({\rm deg}) \in \mathbb{N}$, $\beta_0({\rm deg}) \in (0,1)$ are computable from the rational homotopy group represented by ${\rm deg}$. This extends earlier work by Van Schaftingen and the second author on the Hopf degree to the Novikov's integral representation for rational homotopy groups as developed by Sullivan, Novikov, Hardt and Rivi\`ere.
\end{abstract}

\maketitle
\tableofcontents
\sloppy

\section{Introduction}
Let $\n \subset \R^M$ be a smooth simply connected compact manifold without boundary and $N \in \N$. We denote $\pi_{N}(\n)$ the $N$-th homotopy group of $\n$. 

In \cite{Gromov99}, Gromov introduced the notion of quantitative homotopy theory, which very roughly could be described like this: given a continuous map $\varphi: \S^N \to \n$ can we find a formula which estimates the element in $\pi_{N}(\n)$ it represents by using only analytic estimates of the map $\varphi$?

But even for a simple manifold $\S^n$, its homotopy groups are nontrivial and very difficult to predict in higher degree.
In this generality this question is very difficult, so we are going to restrict our attention here to rational homotopy groups and estimates in fractional Sobolev spaces.

A rational homotopy (sub-)group of $\pi_{N}(\n)$ in the sense of Sullivan \cite{S77} is a homomorphism
\[
 {\d}: \pi_{N}(\n) \to \R.
\]
We are going to identify any such rational homotopy subgroup $\d$ with its induced map acting on the continuous maps $\S^N \to \n$,
\[
 \d: C^0(\S^N,\n) \to \R, \quad \d(\varphi) := \d\brac{[\varphi]},
\]
where $[\varphi] \in \pi_{N}(\n)$ is the class of maps homotopic to $\varphi$. We stress that in particular  ${\d}(\varphi) = 0$ for any map $\varphi \in C^0(\S^N,\n)$ which is constant or homotopic to a constant.

There are two fundamental examples of rational homotopy groups: the classical degree between spheres
\[
 {\d}_{\S^N}: \pi_{N}(\S^N) \to \Z,
\]
and the Hopf degree \cite{H31}, see also \cite[\textsection 18]{BT82},
\[
 {\d}_{H}: \pi_{4N-1}(\S^{2N}) \to \Z.
\]

For $\beta \in (0,1)$ and $p \in [1,\infty)$ the fractional Sobolev space $W^{\beta,p}(\S^N,\R^M)$ consists of all $\varphi \in L^p(\S^N,\R^M)$ such that
\[
 [\varphi]_{W^{\beta,p}(\S^N)} := \brac{\int_{\S^N} \int_{\S^N} \frac{|\varphi(x)-\varphi(y)|^p}{|x-y|^{N+\beta p}}\, dx\, dy}^{\frac{1}{p}} < \infty.
\]
While we will focus on estimates in fractional Sobolev spaces, notice that any $\alpha$-H\"older continuous map $\varphi \in C^\alpha(\S^N,\n)$ belongs to $W^{\beta,p}$ for any $\beta \in (0,\alpha)$ and $p \in [1,\infty)$, so our considerations include the H\"older-continuous case.

A natural question related quantitative homotopy theory is: Given a rational homotopy group $\d: \pi_{N}(\n) \to \R$, can we prove
\begin{equation}\label{eq:degreeestgoal}
 \d (\varphi) \leq [\varphi]_{W^{\beta,p}(\S^N)}^{q}
\end{equation}
for some $q \in [0,\infty)$? The answer is no if $p < \frac{N}{\beta}$ and $\deg$ is nontrivial. Indeed, for any $p < \frac{N}{\beta}$  by scaling arguments one can construct a map $\varphi$ with small norm $[\varphi]_{W^{\beta,p}(\S^N)}$ but with $\d(\varphi) \neq 0$. Consequently, we focus on the borderline case $p = \frac{N}{\beta}$. This is the case where any map $\varphi \in W^{\beta,\frac{N}{\beta}}(\S^N,\R^M)$ belongs to BMO (even VMO), cf \eqref{eq:smallbmo}, and \emph{qualitatively} BMO controls homotopy, as was obtained in the celebrated works \cite{SU82,BN95,BN96}, see also \Cref{la:smallnohomo}.

The question at hand is motivated by a question posed by Van Schaftingen in \cite{VS20}. He showed that the number of homotopy classes to which a map $\varphi: \S^N \to \n$ can belong can be estimated by its $W^{\beta,N/\beta}(\S^N)$-seminorm for any given $\beta > 0$, however without obtaining a power estimate for some finite $q$ as in \eqref{eq:degreeestgoal}, rather he obtained an double exponential-type estimate. As a particular example, he mentioned the question if it was possible to extend the seminal work by Bourgain-Brezis-Mironescu, \cite{BBM05}, to the Hopf degree. Bourgain-Brezis-Mironescu's work shows that in the case of $\pi_{N}(\S^N)$ an estimate of the form \eqref{eq:degreeestgoal} is true for \emph{any} $\beta \in (0,1]$  and $p = \frac{N}{\beta}$, and they obtained a sharp exponent $q$. It is unknown if the same is true for the Hopf degree -- however, in
\cite{SVS20} Van Schaftingen and the second author were able to obtain an estimate for the Hopf degree as in \eqref{eq:degreeestgoal} for any $\beta \geq \beta_0$ and $p = \frac{N}{\beta}$ with sharp exponent $q$, where $\beta_0$ is a computable threshold (which for large $N$ is close to $1$). Let us mention that in the realm of Lipschitz and H\"older continuous maps, estimates for Lipschitz maps imply corresponding estimates for H\"older maps by approximation, cf. \Cref{la:LiptoCalpha}. This is quite different in the category of Sobolev maps.

In this work, we extend the arguments of \cite{SVS20} from the Hopf degree to general rational homotopy groups ${\d}: \pi_{N}(\n) \to \R$. The following is our main result.

\begin{theorem}\label{th:main}
Let ${\d}: \pi_{N}(\n)\to \R$ represent a rational homotopy group. Then there exist two numbers $L = L({\d})
\in \{0,1,\ldots,N-2\}$ and 
$\beta_0 > 0$
and a constant $C=C({\d})$ such that the following holds for any $\beta > \beta_0$:

Let $f \in \lip(\S^N,\n)$ then 
\begin{equation}\label{eq:main:estW}
 |{\d}([f])| \leq C\, [f]_{W^{\beta,\frac{N}{\beta}}(\S^N)}^{\frac{N+L}{\beta}}.
\end{equation}
\end{theorem}
 $L$ is the number obtained from the corresponding tree-graph of $\d$, cf. \Cref{pr:hardtriviere}; but let us mention that $L=0$ in the case of usual degree $\S^N \to \S^N$ and $L=1$ in the case of the Hopf degree $\S^{4N-1} \to \S^{2N}$. 

As previously discussed, \Cref{th:main} extends the estimates for the Hopf degree in \cite{SVS20} -- it was shown there that the exponent $\frac{N+L}{\beta}$ is sharp in that case. See also \Cref{s:examples}. What is unclear, however, is whether the lower bound $\beta_0$ is sharp -- and there are substantial indications it is not: for the degree ${\d}_{\S^N}: \pi_{N}(\S^N) \to \Z$, the already mentioned seminal work by Bourgain-Brezis-Mironescu \cite{BBM05} shows that $\beta_0$ can be taken zero in that case. Slightly extending the question from \cite{VS20,SVS20} we could then ask

\begin{question}
Can we choose $\beta_0 = 0$ in \Cref{th:main}?
\end{question}

One of the main ingredients for the proof of \Cref{th:main} is an integral representation formula due to Hardt and Rivi\`ere \cite{HR08}, which we will describe in \Cref{pr:hardtriviere}, combined with Harmonic Analysis estimates that we describe in the proposition below. 

We need some notation:

A tree-graph $T$ is a connected, simply connected and oriented planar graph.
$T$ is oriented in the sense that, all edges have direction and at each vertex $A$, except one, there is only one segment leaving.
The vertex without leaving segment is at the top of the graph $T$ where all attached segments are arriving.
For each vertex $A$ of $T$ we can form a sub tree-graph $T_A$ such that its top vertex is $A$ and its other vertices and segments are those of $T$ that can arrive $A$.
And at each vertex $A$ of $T$ we assign a closed form $\omega_A \in \lip(\Ep^\ast \n)$.

\begin{figure}
\begin{tikzpicture}[scale = 1,thick, arr/.style={postaction={
      decorate,
      decoration={
    markings,
    mark=at position 0.5 with {\arrow[scale=2]{>}}}}}]
	\coordinate (A_0) at (0.5,1);
	\coordinate (A_1) at (0.5,-0.5);
	\coordinate (A_2) at (-2,-2);
	\coordinate (A_3) at (0,-2);
	\coordinate (A_4) at (2.5,-2);
	\coordinate (A_5) at (-2.8,-3);
	\coordinate (A_6) at (-1,-3);
	\coordinate (A_7) at (1.3,-3);
	\coordinate (A_8) at (3,-3);
	\filldraw[black] (A_0) circle (2pt) node[anchor=east]{$\omega_{A_0}$};
	\filldraw[black] (A_1) circle (2pt) node[anchor=east]{$\omega_{A_1}$};
	\filldraw[black] (A_2) circle (2pt) node[anchor=east]{$\omega_{A_2}$};
	\filldraw[black] (A_3) circle (2pt) node[anchor=north]{$\omega_{A_3}$};
	\filldraw[black] (A_4) circle (2pt) node[anchor=west]{$\omega_{A_4}$};
	\filldraw[black] (A_5) circle (2pt) node[anchor=north]{$\omega_{A_5}$};
	\filldraw[black] (A_6) circle (2pt) node[anchor=north]{$\omega_{A_6}$};
	\filldraw[black] (A_7) circle (2pt) node[anchor=north]{$\omega_{A_7}$};
	\filldraw[black] (A_8) circle (2pt) node[anchor=north]{$\omega_{A_8}$};
	\draw[arr] (A_1) -- (A_0);
	\draw[arr] (A_2) -- (A_1);
	\draw[arr] (A_3) -- (A_1);
	\draw[arr] (A_4) -- (A_1);
	\draw[arr] (A_5) -- (A_2);
	\draw[arr] (A_6) -- (A_2);
	\draw[arr] (A_7) -- (A_4);
	\draw[arr] (A_8) -- (A_4);
\end{tikzpicture}
\caption{An example of tree-graph $T$} \label{tree-graph 1}
\end{figure}
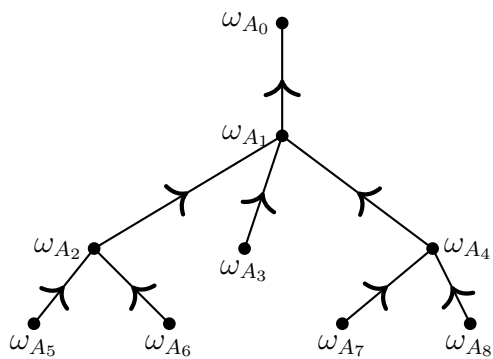

Given a degree $\d : \pi_{N}(\n) \to \R$ and $f : \S^{N} \to \n$, we assign a signed sum of tree-graphs $T = \sum (-1)^{n_i} T_i$, where $T_i$ are tree graphs.
Then
\begin{equation} \label{integral formula}
 {\d}([f]) = \int_{\S^N} f^\ast (T)
 \end{equation}
  where the tree form $f^\ast (T)$ is obtained inductively as follows.
First, if $T$ is a single vertex $A$, then $f^\ast (T) = f^\ast \omega_A$.
For a smooth $\ell$-form $\eta \in \lip(\Ep^\ell \S^{N})$, $\ell \in \{1,\ldots,N-1\}$ we set
\[
 d^{-1} \eta = d^\ast\lap^{-1} \eta
\]
where $\lap$ denotes the Laplace-DeRham operator for $\ell$-forms (which is invertible since the $\ell$-th and $(N-\ell)$-th De Rham cohomology group of $\S^N$ is trivial) and $d^\ast$ is the co-differential.

Then $f^\ast (T)$ is defined by
\[
\begin{split}
f^\ast (T) &= \sum (-1)^{n_i} f^\ast (T_i), \\
f^\ast (T_i) &= f^* \omega_{A_i} \wedge \bigwedge_{j} d^{-1} f^* (T_{ij})
\end{split}
\]
where $A_i$ is the top vertex of $T_i$, $A_{ij}$ are vertices directly arriving $A_i$ and $T_{ij}$ are sub tree-graphs of $T_i$ having $A_{ij}$ as the top vertices, and $f^\ast (T_{ij})$ are defined inductively.

For example, in the Figure \ref{tree-graph 1}, denoting $T_i$ as the sub tree-graph with top vertex $A_i$,
\[
\begin{split}
f^\ast (T) &= f^\ast (\omega_{A_0}) \wedge d^{-1} f^\ast (T_1)\\
f^\ast (T_1) &= f^\ast (\omega_{A_1}) \wedge d^{-1} f^\ast (T_2) \wedge d^{-1} f^\ast (\omega_{A_3}) \wedge d^{-1} f^\ast (T_4)\\
f^\ast (T_2) &= f^\ast (\omega_{A_2}) \wedge d^{-1} f^\ast (\omega_{A_5}) \wedge d^{-1} f^\ast (\omega_{A_6})\\
f^\ast (T_4) &= f^\ast (\omega_{A_4}) \wedge d^{-1} f^\ast (\omega_{A_7}) \wedge d^{-1} f^\ast (\omega_{A_8}).
\end{split}
\]
Note that the order of wedge product is left-to-right according to the tree-graph.

\begin{proposition}\label{pr:mainharmest}
Assume that $T$ is a tree-graph as above, e.g. in \Cref{tree-graph 1}, with $L$ vertical arrows. 
Also assume on each vertex of $T$, the assigned form is at most $M_{max}$ form.
Then there exists a $\beta_0 > 0$  such that for any $\beta > \beta_0$ the following estimate holds.

For a constant $C = C(\beta,T)$, we have for any $f \in \lip(\S^N,\R^M)$,
\[
  \abs{\int_{\S^N} f^\ast(T)}  \leq C\, \brac{ [f]_{W^{\beta,\frac{N}{\beta}}(\S^N)}^{\frac{N+L}{\beta}}+1}.
\]

\end{proposition}
Let us remark on one possible application that will be the subject of future research. As in \cite{HST14}, see also \cite{S20Aeq,VS20,HMS29}, one can use the integral formulas for degree, Hopf degree, or more generally rational homotopy groups to define a ``homotopy invariant'' for maps $f: \S^N \to \R^M$ with a rank restriction $\rank Df \leq K$. Since our analysis is based mostly on harmonic analysis, it applies also to this case -- as was discussed in \cite[Proposition 4.3.]{VS20}. In \cite{HST14} this was used to obtain nontrivial Lipschitz topology results for the Heisenberg group, in \cite{HMS29} this will be used to do the same for H\"older-topology results. One of the main motivations of this work is to provide via the framework of rational homotopy group a large number of topological invariants which possibly could be applied in Heisenberg groups or more generally Carnot groups. 

The remainder of this paper is structured as follows. In \Cref{s:proofmainharmest} we prove the crucial estimate \Cref{pr:mainharmest}. Using the Hardt--Rivi\`ere representation formula for rational homotopy groups, \Cref{pr:hardtriviere}, we discuss in \Cref{s:proofmain} how to conclude the proof of \Cref{th:main} from \Cref{pr:mainharmest}. Lastly, in \Cref{s:examples} we show in the concrete examples developed in \cite{HR08} how our estimates extend theirs.

\subsection*{Acknowledgment} A.S. is funded by Simons foundation, grant no 579261, and NSF Career DMS-2044898.

\section{Harmonic Analysis Estimates: Proof of Proposition~\ref{pr:mainharmest}}\label{s:proofmainharmest}
In order to prove \Cref{pr:mainharmest} we first extend the result in \cite[Proposition 3.2.]{SVS20}.

\begin{lemma}\label{la:lapmsest}

Let $\alpha_0 \in (0,\frac{M_0}{M_0+1}]$, $1 \leq M_0 < N$ and $f \in \lip(\R^N,\R^M)$ with compact support. For any smooth $M_0$-form $\omega$, 
for any $\beta > 1-\frac{\alpha_0}{M_0}$ we have for any $\alpha \in [\alpha_0,M_0)$,
we have
\[
 \| \lapms{\alpha}f^\ast(\omega)\|_{L^{\frac{N}{M_0-\alpha}}(\R^N)} \leq C\, \brac{[f]_{W^{\beta,\frac{N}{\beta}}}^{\frac{M_0}{\beta}} +1}
\]
where the constant $C > 0$ depends on $\|f\|_{L^\infty}$, $\omega$, $\alpha$, and $\beta$, and the support of $f$.

Here for $\alpha \in (0,N)$ and \emph{functions} $G: \R^N \to \R$
\[
 \lapms{\alpha} G(x) := \int_{\R^N} |x-y|^{\alpha - N} G(y)\, dy
\]
denotes the Riesz potential. It acts on forms component-wise.
\end{lemma}

\begin{proof}
By Sobolev embedding, for $\alpha \geq \alpha_0$,
\[
\| \lapms{\alpha}f^\ast(\omega)\|_{L^{\frac{N}{M_0-\alpha}}(\R^N)}  \aleq \| \lapms{\alpha_0}f^\ast(\omega)\|_{L^{\frac{N}{M_0-\alpha_0}}(\R^N)}, 
\]
so we may assume $\alpha = \alpha_0$.

First we may assume that $\omega = \tilde{h} \theta_1 \wedge \ldots \wedge \theta_{M_0}$,
where $\theta_i = dp^{j_i}$ are closed $1$-forms and $h := \tilde{h} \circ f \in \lip(\R^N)$. 

Take (for now) any $\gamma \in (0, \alpha_0]$, and denote $p_\gamma := \frac{N}{M_0-\gamma}$ and $p_0 := \frac{N}{M_0-\alpha_0}$. Observe $p_\gamma \in (1,p_0]$ and
\[
\frac{Np_\gamma}{N+\gamma p_\gamma} = \frac{N}{M_0} \in (1,\infty).
\]
By duality, there exists a test-form $\varphi \in C_c^\infty(\Ep^{N-M_0} \R^{N})$ with $\|\varphi\|_{L^{p_\gamma'}(\R^N)} \equiv \|\varphi\|_{L^{\frac{N}{N-M_0+\gamma}}(\R^N)} \leq 1$, such that, using Sobolev embedding in the first inequality,
\[ 
\begin{split}
\| \lapms{\alpha_0}f^\ast(\omega)\|_{L^{p_0}(\R^N)} \aleq \| \lapms{\gamma}f^\ast(\omega)\|_{L^{p_\gamma}(\R^N)} \aleq& \int_{\R^N} \tilde{h} \circ f\, f^\ast(\theta_1) \wedge \ldots f^\ast(\theta_{M_0})\wedge \lapms{\gamma} \varphi\\
=& \int_{\R^N} h\, f^\ast(\theta_1) \wedge \ldots f^\ast(\theta_{M_0})\wedge \psi
\end{split}
\]
where we set $\psi := \lapms{\gamma} \varphi$. Observe that by Sobolev embedding,
\[
 \|\psi\|_{L^{\frac{N}{N-M_0}}(\R^N)} = \|\lapms{\gamma} \varphi\|_{L^{\frac{N p_\gamma'}{N-\gamma p_\gamma'}}(\R^N)} \aleq \|\varphi\|_{L^{p_\gamma'}(\R^N)} \leq 1.
\]

Denote $H$, $F$, $\Psi$ the harmonic extensions to $\R^{N+1}_+$ of $h$, $f$, $\psi$. We find from Stokes' theorem (using that $d \theta_i = 0$)
\begin{equation}\label{eq:west:234234}
\begin{split}
 &\| \lapms{\gamma}f^\ast(\omega)\|_{L^{p_\gamma}(\R^N)} \\
 \aleq& \abs{\int_{\R^{N+1}_+} d\brac{H\, F^\ast(\theta_1) \wedge \ldots F^\ast(\theta_{M_0})\wedge \Psi}}\\
 \aleq& \abs{\int_{\R^{N+1}_+} {dH\wedge F^\ast(\theta_1) \wedge \ldots F^\ast(\theta_{M_0})\wedge \Psi}}+\abs{\int_{\R^{N+1}_+} {H\, F^\ast(\theta_1) \wedge \ldots F^\ast(\theta_{M_0})\wedge d\Psi}}\\
 \aleq& \int_{\R^{N+1}_+} |DH|\, |DF|^{M_0} |\Psi|+ \|H\|_{L^\infty}\, \int_{\R^{N+1}_+} |DF|^{M_0} |D\Psi|.
 \end{split}
\end{equation}

We estimate the \underline{first term on the right-hand side of \eqref{eq:west:234234}}.

By the representation formula for harmonic functions
\[
 |\Psi(x,t)| \aleq \mathcal{M} \psi(x),
\]
where $\mathcal{M}$ is the Hardy-Littlewood maximal function.

Moreover, see e.g. \cite[Proposition 10.2]{LS20}, since $1- \frac{\gamma}{M_0} \in (0,1)$,
\[
 \brac{\int_{\R^{N+1}_+} \brac{t^{\gamma-\frac{1}{p_\gamma}}  |DF|^{M_0}}^{p_\gamma} }^{\frac{1}{p_\gamma}} \aeq [f]_{W^{1-\frac{\gamma}{M_0},M_0 p_\gamma}}^{M_0}.
\]
And from \cite[Theorem 10.8]{LS20}, since $\gamma \in (0,1)$,
\[
 \brac{\int_{\R^N} \brac{\int_{\R_+} \brac{t^{-\gamma+\frac{1}{p_\gamma}} |DH|}^{\frac{N}{N-M_0+\gamma}} dt}^{\frac{N-M_0+\gamma}{\gamma}} dx}^{\frac{\gamma}{N}} \aeq \|h\|_{\dot{F}^{\gamma}_{\frac{N}{\gamma},\frac{N}{N-M_0+\gamma}}}
\]
where $\dot{F}^{\gamma}_{\frac{N}{\gamma},\frac{N}{N-M_0+\gamma}}$ is the Triebel-Lizorkin space; see also the presentation in \cite{Ingmanns20}.

Then we have (here we use $M_0 > \gamma$ which is true since $M_0 > \alpha \geq \gamma$)
\[
\begin{split}
  &\int_{\R^{N+1}_+} |DH|\, |DF|^{M_0} |\Psi|\\
 \aleq& [f]_{W^{1-\frac{\gamma}{M_0},M_0 p_\gamma }}^{M_0}\brac{\int_{\R^{N}} \abs{\mathcal{M}\psi(x)}^{\frac{N}{N-M_0 + \gamma}} \int_{\R_+}\brac{t^{-\gamma+\frac{1}{p_\gamma}}  |DH|}^{\frac{N}{N-M_0 + \gamma} } dt\, dx} ^{\frac{N-M_0+\gamma}{N}} \\
 \aleq& [f]_{W^{1-\frac{\gamma}{M_0},M_0 p_\gamma }}^{M_0}
 \|\psi\|_{L^{\frac{N}{N-M_0}}(\R^N)} 
 \brac{\int_{\R^{N}} \brac{\int_{\R_+}\brac{t^{-\gamma+\frac{1}{p_\gamma}}  |DH|}^{\frac{N}{N-M_0 + \gamma} } dt}^{\frac{N-M_0+\gamma}{\gamma}} dx} ^{\frac{\gamma}{N}} \\
 \aleq&[f]_{W^{1-\frac{\gamma}{M_0},M_0 p_\gamma }}^{M_0} \|h\|_{F^{\gamma}_{\frac{N}{\gamma},\frac{N}{N-M_0+\gamma}}}.
 \end{split}
\]
Let now $\beta > \gamma$.
We use the Gagliardo-Nirenberg estimate for Triebel-Lizorkin spaces \cite[Proposition 5.6]{BM18} and obtain for any $\tilde{\gamma} < \gamma < \beta$, taking $\theta \in (0,1)$ such that $\gamma = \theta \beta + (1-\theta) \tilde{\gamma}$,
\[
\begin{split}
 \|h\|_{\dot{F}^{\gamma}_{\frac{N}{\gamma},\frac{N}{N-M_0+\gamma}}} \aleq& \|h\|_{\dot{F}^{\beta}_{\frac{N}{\beta},\frac{N}{\beta}}}^\theta \|h\|_{\dot{F}^{\tilde{\gamma}}_{\frac{N}{\tilde{\gamma}},\frac{N}{\tilde{\gamma}}}}^{1-\theta} \\
 \aeq& [h]_{W^{\beta,\frac{N}{\beta}}}^\theta\, [h]_{W^{\tilde{\gamma},\frac{N}{\tilde{\gamma}}}}^{1-\theta} \\
 \aleq& [h]_{W^{\beta,\frac{N}{\beta}}}^\theta\, \|h\|_{L^\infty}^{(1-\theta)\frac{\tilde{\gamma}}{N} (\frac{N}{\tilde{\gamma}}-\frac{N}{\beta})} [h]_{W^{\beta,\frac{N}{\beta}}}^{(1-\theta)\frac{\tilde{\gamma}}{\beta}} \\
 =& [h]_{W^{\beta,\frac{N}{\beta}}}^{\frac{\gamma}{\beta}}\, \|h\|_{L^\infty}^{1-\frac{\gamma}{\beta}} .
 \end{split}
\]
If we additionally we assume $\beta > 1-\frac{\gamma}{M_0}$, then the above estimates lead to
\[
\begin{split}
 \int_{\R^{N+1}_+} |DH|\, |DF|^{M_0} |\Psi| \aleq &[f]_{W^{1-\frac{\gamma}{M_0},M_0 p_\gamma}}^{M_0}  [f]_{W^{\beta,\frac{N}{\beta}}}^{\frac{\gamma}{\beta}}\\
 \aleq&[f]_{W^{\beta,\frac{N}{\beta}}}^{M_0 \frac{1-\frac{\gamma}{M_0}}{\beta}}  [f]_{W^{\beta,\frac{N}{\beta}}}^{\frac{\gamma}{\beta}}\\
 =&[f]_{W^{\beta,\frac{N}{\beta}}}^{\frac{M_0}{\beta}}.
 \end{split}
\]
Here the constants depend on $\|f\|_{L^\infty}$.
In a similar fashion we estimate the \underline{second term on the right-hand side of \eqref{eq:west:234234}}.

By the maximum principle,
\[
 \|H\|_{L^\infty} \aleq \|h\|_{L^\infty} .
\]
Moreover, for any $\beta \in (0,1)$, again by \cite[Proposition 10.2]{LS20},
\[
 \brac{\int_{\R^{N+1}_+}  \brac{  t^{M_0-\frac{M_0 \beta}{N} - M_0 \beta} |DF|^{M_0} }^{\frac{N}{M_0 \beta}} }^{M_0 \frac{\beta}{N}} \aeq [f]_{W^{\beta,\frac{N}{\beta}}}^{M_0}.
\]
Also, whenever $\beta \in (0,1)$ satisfies $(1-\beta)M_0 < \gamma$,
\[
\begin{split}
 \brac{\int_{\R^{N+1}_+}  \brac{  t^{1-\frac{N-M_0 \beta}{N}-M_0 (1-\beta)} |D\Psi|}^{\frac{N}{N-M_0 \beta}}}^{\frac{N-M_0 \beta}{N}} \aeq& [\psi]_{W^{M_0(1-\beta),\frac{N}{N-M_0 \beta}}}\\
  =& [\lapms{\gamma}\varphi]_{W^{M_0(1-\beta),\frac{N}{N-M_0 \beta}}} \\
 \aeq&[\lapms{\gamma}\varphi]_{\dot{F}^{M_0(1-\beta)}_{\frac{N}{N-M_0 \beta},\frac{N}{N-M_0 \beta}}}\\
 =&[\lapms{-M_0(1-\beta)+\gamma}\varphi]_{\dot{F}^{0}_{\frac{N}{N-M_0 \beta},\frac{N}{N-M_0 \beta}}}\\
 \aleq&[\varphi]_{\dot{F}^{0}_{p_{\gamma'},2}}\aeq [\varphi]_{L^{p_{\gamma'}}}\\
 \aleq& 1.
 \end{split}
\]
With these observations,
\[
\begin{split}
 \int_{\R^{N+1}_+} |DF|^{M_0} |D\Psi|=&\int_{\R^{N+1}_+} t^{M_0-\frac{M_0 \beta}{N} - M_0 \beta}|DF|^{M_0}\, t^{1-\frac{N-M_0 \beta}{N}-M_0 (1-\beta)}|D\Psi|\\
 \aleq&\brac{\int_{\R^{N+1}_+}  \brac{  t^{M_0-\frac{M_0 \beta}{N} - M_0 \beta} |DF|^{M_0} }^{\frac{N}{M_0 \beta}} }^{M_0 \frac{\beta}{N}} \\
 &\quad \cdot \brac{\int_{\R^{N+1}_+}  \brac{  t^{1-\frac{N-M_0 \beta}{N}-M_0 (1-\beta)} |D\Psi|}^{\frac{N}{N-M_0 \beta}}}^{\frac{N-M_0 \beta}{N}}\\
 \aleq& [f]_{W^{\beta,\frac{N}{\beta}}}^{M_0}.
 \end{split}
\]

In conclusion, we have shown
\[
 \| \lapms{\alpha_0}f^\ast(\omega)\|_{L^{\frac{N}{M_0-\alpha_0}}(\R^N)} \aleq  [f]_{W^{\beta,\frac{N}{\beta}}}^{\frac{M_0}{\beta}} + [f]_{W^{\beta,\frac{N}{\beta}}}^{M_0} \aleq [f]_{W^{\beta,\frac{N}{\beta}}}^{\frac{M_0}{\beta}} + 1.
\]
The above holds whenever the following conditions are met by $\gamma \in (0,1)$ and $\beta \in (0,1)$:
\begin{itemize}
 \item $0 < \gamma \leq \alpha_0$
 \item $(1-\beta) M_0 < \gamma$
 \item $\beta > \gamma$.
\end{itemize}
That is 
\[
 \beta > \max\{ \frac{M_0-\gamma}{M_0}, \gamma\}.
\]
That is, we can make this argument work whenever 
\[
\begin{split}
 \beta > \inf_{\gamma \in (0,\alpha_0]} \max\{ \frac{M_0-\gamma}{M_0}, \gamma\} =&  \begin{cases}
                                                                                    \frac{M_0}{M_0+1} \quad &\text{if }\alpha_0\geq \frac{M_0}{M_0+1} \\
                                                                                    1 - \frac{\alpha_0}{M_0}\quad & \text{if }\alpha_0< \frac{M_0}{M_0+1} 
                                                                                  \end{cases}\\
%
\end{split}
\]
\end{proof}

To extend \Cref{la:lapmsest} from an estimate for $f^\ast(\omega)$ to an estimate of $f^\ast(T)$, we observe the following iterative estimate which essentially follows from the fractional Leibniz rule.

\begin{lemma}\label{la:lapmsest2}
Assume that $T$ is a tree as in the beginning of our paper as in \Cref{fig:la:lapmsest2},
\begin{figure}
\begin{center}
\begin{tikzpicture}[scale = 1,thick, arr/.style={postaction={
      decorate,
      decoration={
    markings,
    mark=at position 0.5 with {\arrow[scale=2]{>}}}}}]
	\coordinate (omega_0) at (0.5,-0.5);
	\coordinate (T_1) at (-2,-3);
    \coordinate (T_2) at (-1,-3);
    \coordinate (T_3) at (0,-3);
    \coordinate (T_4) at (1,-3);
    \coordinate (T_5) at (2,-3);
     \coordinate (T_L) at (3,-3);

	\filldraw[black] (omega_0) circle (2pt) node[anchor=east]{$\omega_{0}$};
	\filldraw[black] (T_1) circle (2pt) node[anchor=north]{$T_1$};
	\filldraw[black] (T_2) circle (2pt) node[anchor=north]{$T_2$};
	\filldraw[black] (T_3) circle (2pt) node[anchor=north]{$T_3$};
	\filldraw[black] (T_4) circle (2pt) node[anchor=north]{$T_4$};
	\filldraw[black] (T_5) node[anchor=north]{$\ldots$};
	\filldraw[black] (T_L) circle (2pt) node[anchor=north]{$T_L$};
	\draw[arr] (T_1) -- (omega_0);
	\draw[arr] (T_2) -- (omega_0);
		\draw[arr] (T_3) -- (omega_0);
			\draw[arr] (T_4) -- (omega_0);
				\draw[arr] (T_L) -- (omega_0);
		
\end{tikzpicture}
\end{center}
\caption{\label{fig:la:lapmsest2} Configuration in \Cref{la:lapmsest2}}
\end{figure}
where $T_1,\ldots,T_L$ are themselves again trees, assume that 
\[
 f^\ast(T) \equiv f^\ast(\omega_0) \wedge d^{-1}f^\ast(T_1) \wedge \ldots \wedge d^{-1}f^\ast(T_L)
\]
is an $M$-form, $\omega_0$ is an $M_0$-form, $f^\ast(T_\ell)$ is an $M_{\ell}$ form.

Assume $\alpha \in (0,1]$, then for any $\sigma \in [0,\alpha]$,
\[
\begin{split}
 \| \lapms{\alpha}f^\ast(T)\|_{L^{\frac{N}{M-\alpha}}(\Ep^{M}\R^N)} \aleq&  \|\lapms{\sigma} f^\ast(\omega_0)\|_{L^{\frac{N}{M-\sigma}}(\Ep^{M_0} \R^N)}\, \|\lapms{1-\sigma} f^\ast(T_1)\|_{L^{\frac{N}{M_1-(1-\sigma)}}(\Ep^{M_1} \R^N)}\\
 &\, \ldots\, \|\lapms{1-\sigma} f^\ast(T_L)\|_{L^{\frac{N}{M_L-(1-\sigma)}}(\Ep^{M_L} \R^N)}.
 \end{split}
\]

\end{lemma}
\begin{proof}
By duality, for some $\varphi \in C_c^\infty(\Ep^{M-N} \R^N)$, $\|\varphi\|_{L^{\frac{N}{N+\alpha-M}}(\Ep^{M-N} \R^N)} \leq 1$,
\[
\begin{split}
 \| \lapms{\alpha}f^\ast(T)\|_{L^{\frac{N}{M-\alpha}}(\Ep^{M}\R^N)} 
 \aleq& \int_{\R^N} f^\ast(T) \wedge \lapms{\alpha} \varphi\\
 =&\int_{\R^N} f^\ast(\omega_0) \wedge d^{-1} f^\ast(T_1) \wedge \ldots d^{-1} f^\ast(T_L) \wedge \lapms{\alpha} \varphi\\
 =&\int_{\R^N} \lapms{\sigma} f^\ast(\omega_0) \wedge \laps{\sigma}\brac{d^{-1} f^\ast(T_1) \wedge \ldots d^{-1} f^\ast(T_L) \wedge \lapms{\alpha} \varphi}.\\
 \end{split}
\]
Observe that 
\[
 M = M_0 + \sum_{k=1}^L(M_k-1).
\]
Then, using the fractional Leibniz rule, see e.g. \cite{LS20}, we find
\[
\begin{split}
 &\| \lapms{\alpha}f^\ast(T)\|_{L^{\frac{N}{M-\alpha}}} \\
 \aleq& \|\lapms{\sigma} f^\ast(\omega_0)\|_{L^{\frac{N}{M_0-\sigma}}}\, \|\lapms{\alpha} \varphi\|_{L^{\frac{N}{N-M}}}\,\\
 &\cdot \Big (\\
 & \|\laps{\sigma} d^{-1} f^\ast(T_1) \|_{L^{\frac{N}{M_1 -(1-\sigma)}}}\, \|d^{-1} f^\ast(T_2) \|_{L^{\frac{N}{M_2 -1}}} \ldots \|d^{-1} f^\ast(T_L)\|_{L^{\frac{N}{M_L -1}}}\\
 &+\|d^{-1} f^\ast(T_1) \|_{L^{\frac{N}{M_1 -1}}}\, \|\laps{\sigma} d^{-1} f^\ast(T_2) \|_{L^{\frac{N}{M_2 -(1-\sigma)}}} \ldots \|d^{-1} f^\ast(T_L)\|_{L^{\frac{N}{M_L -1}}}\\
 &+ \ldots + \|d^{-1} f^\ast(T_1) \|_{L^{\frac{N}{M_1 -1}}}\, \|d^{-1} f^\ast(T_2) \|_{L^{\frac{N}{M_2 -1}}} \ldots \|\laps{\sigma} d^{-1} f^\ast(T_L)\|_{L^{\frac{N}{M_L -(1-\sigma)}}}\\
 & \Big )\\
 &+ \|\lapms{\sigma} f^\ast(\omega_0)\|_{L^{\frac{N}{M_0-\sigma}}}\, \, \|\lapms{\alpha-\sigma} \varphi\|_{L^{\frac{N}{N+\sigma-M}}}\\
 &\cdot \Big (
 \|d^{-1} f^\ast(T_1) \|_{L^{\frac{N}{M_1 -1}}}\, \|d^{-1} f^\ast(T_2) \|_{L^{\frac{N}{M_2 -1}}} \ldots \|d^{-1} f^\ast(T_L)\|_{L^{\frac{N}{M_L -1}}}
  \Big ).
 \end{split}
\]
By Sobolev embedding, and using that we can write $d^{-1} = d^\ast \lap^{-1} = \lapms{1} R$ where $R$ is a zero-multiplier (a combination of Riesz transforms) and thus bounded on $L^q$ for any $q \in (1,\infty)$, 
\[
\begin{split}
 \aleq&\|\lapms{\sigma} f^\ast(\omega_0)\|_{L^{\frac{N}{M_0-\sigma}}}\, \|\varphi\|_{L^{\frac{N}{N+\alpha-M}}}\, \\
 &\cdot \Big (\\
 & \|\lapms{1-\sigma}f^\ast(T_1) \|_{L^{\frac{N}{M_1 -(1-\sigma)}}}\, \|\lapms{1-\sigma} f^\ast(T_2) \|_{L^{\frac{N}{M_2 -(1-\sigma)}}} \ldots \|\lapms{1-\sigma} f^\ast(T_L)\|_{L^{\frac{N}{M_L -(1-\sigma)}}}\\
 & \Big )\\
 \end{split}
\]
This proves the claim.
\end{proof}
Now we can obtain a version of \Cref{la:lapmsest}, only for trees $T$. Observe that the best choice of $\beta_0$ is computable by combinatorial observations, in particular for every $T$ it is easy to compute the best $\beta_0$.
But here we give an easy choice of $\beta_0$.
\begin{lemma}\label{la:lapmsest3}
Let $1\leq M_0 < N$ and $f \in \lip(\R^N,\R^M)$ with compact support.

Let $T$ be a tree, such that $f^\ast(T)$ is an $M_0$-form and on each vertex of $T$, the assigned form is at most $M_{max}$ form.
Assume $\alpha_0 \in (0,\frac{M_0}{M_0+1}]$.
Then for any $\alpha \in [\alpha_0, M_0)$ there exists $\beta_0 = \beta(\alpha,T) > 0$ such that for any $\beta \in (\beta_0,1]$:
\[
 \| \lapms{\alpha}f^\ast(T)\|_{L^{\frac{N}{M_0-\alpha}}(\R^N)} \leq C(T,\alpha_0)\, [f]_{W^{\beta,\frac{N}{\beta}}}^{\frac{M_0+L }{\beta}} +1.
\]
Here $L$ is the number of vertical lines in $T$ (or equivalently the number of $d^{-1}$ in the formula of $f^\ast (T)$).
\end{lemma}

\begin{proof}
Since by Sobolev embedding, for $\alpha \geq \alpha_0$,
\[
  \| \lapms{\alpha}f^\ast(T)\|_{L^{\frac{N}{M_0-\alpha}}(\R^N)} \aleq  \| \lapms{\alpha_0}f^\ast(T)\|_{L^{\frac{N}{M_0-\alpha_0}}(\R^N)}
\]
so we may assume that $\alpha = \alpha_0$.

If $T$ consists of exactly one $M_0$-form $\omega_0$, then the claim follows directly from \Cref{la:lapmsest}.

{\underline {Case 1}}

Consider the case when $f^\ast (T) = f^\ast (\omega) \wedge d^{-1} f^\ast (\omega_1) \wedge \ldots \wedge d^{-1}f^\ast (\omega_L)$ where $\omega$ is an $M$ form and $\omega_i$ is an $M_i$ form.
Similar in \Cref{la:lapmsest2}, for $\alpha_0 \in (0,\frac{M_0}{M_0+1}]$ and for any $\alpha_1 \in [0,\alpha_0]$,
\[
\begin{split}
&\| \lapms{\alpha_0}f^\ast(T)\|_{L^{\frac{N}{M_0-\alpha_0}}(\R^N)}\\
\aleq & \|\lapms{\alpha_1} f^\ast(\omega)\|_{L^{\frac{N}{M-\alpha_0}}}\, \|\lapms{1-\alpha_1} f^\ast(\omega_1)\|_{L^{\frac{N}{M_1-(1-\alpha_1)}}}\, \ldots\, \|\lapms{1-\alpha_1} f^\ast(\omega_L)\|_{L^{\frac{N}{M_L-(1-\alpha_1)}}}\\
\aleq & \brac{[f]_{W^{\beta,\frac{N}{\beta}}}^{\frac{M}{\beta}} + 1}\, \brac{[f]_{W^{\beta,\frac{N}{\beta}}}^{\frac{M_1}{\beta}} + 1}\, \ldots\,\brac{[f]_{W^{\beta,\frac{N}{\beta}}}^{\frac{M_L}{\beta}} + 1}
\end{split}
\]
whenever $\beta \in (0,1)$ is large enough.

{\underline {Case 2}}

Next we assume $T$ has the structure 
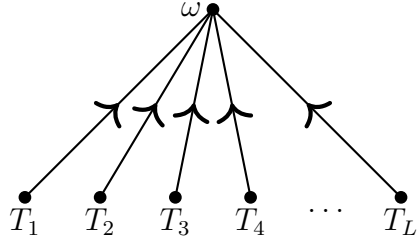
\begin{center}
\begin{tikzpicture}[scale = 1,thick, arr/.style={postaction={
      decorate,
      decoration={
    markings,
    mark=at position 0.5 with {\arrow[scale=2]{>}}}}}]
	\coordinate (omega_0) at (0.5,-0.5);
	\coordinate (T_1) at (-2,-3);
    \coordinate (T_2) at (-1,-3);
    \coordinate (T_3) at (0,-3);
    \coordinate (T_4) at (1,-3);
    \coordinate (T_5) at (2,-3);
     \coordinate (T_L) at (3,-3);

	\filldraw[black] (omega_0) circle (2pt) node[anchor=east]{$\omega$};
	\filldraw[black] (T_1) circle (2pt) node[anchor=north]{$T_1$};
	\filldraw[black] (T_2) circle (2pt) node[anchor=north]{$T_2$};
	\filldraw[black] (T_3) circle (2pt) node[anchor=north]{$T_3$};
	\filldraw[black] (T_4) circle (2pt) node[anchor=north]{$T_4$};
	\filldraw[black] (T_5) node[anchor=north]{$\ldots$};
	\filldraw[black] (T_L) circle (2pt) node[anchor=north]{$T_L$};
	\draw[arr] (T_1) -- (omega_0);
	\draw[arr] (T_2) -- (omega_0);
		\draw[arr] (T_3) -- (omega_0);
			\draw[arr] (T_4) -- (omega_0);
				\draw[arr] (T_L) -- (omega_0);
		
\end{tikzpicture}
\end{center}

Here $\omega$ is an $M$-form ($M < M_0$) and $T_1,\ldots,T_L$ create $M_1,\ldots,M_L$-forms via the pullback $f^\ast$, respectively.

We get from \Cref{la:lapmsest} and \Cref{la:lapmsest2} 
\[
\begin{split}
  &\| \lapms{\alpha_0}f^\ast(T)\|_{L^{\frac{N}{M_0-\alpha_0}}(\R^N)}\\
   \aleq & \|\lapms{\alpha_1} f^\ast(\omega)\|_{L^{\frac{N}{M-\alpha_1}}}\, \|\lapms{1-\alpha_1} f^\ast(T_1)\|_{L^{\frac{N}{M_1-(1-\alpha_1)}}(\Ep^{M_1} \R^N)}\, \ldots\, \|\lapms{1-\alpha_1} f^\ast(T_L)\|_{L^{\frac{N}{M_L-(1-\alpha_1)}}(\Ep^{M_L} \R^N)}\\
 \aleq & \brac{[f]_{W^{\beta,\frac{N}{\beta}}}^{\frac{M}{\beta}} + 1}\, \|\lapms{1-\alpha_1} f^\ast(T_1)\|_{L^{\frac{N}{M_1-(1-\alpha_1)}}(\Ep^{M_1} \R^N)}\, \ldots\, \|\lapms{1-\alpha_1} f^\ast(T_L)\|_{L^{\frac{N}{M_L-(1-\alpha_1)}}(\Ep^{M_L} \R^N)},
   \end{split}
\]
whenever $\alpha_1 \in [0,\alpha_0]$ and $\beta > 1 - \frac{\alpha_1}{M}$.

Suppose for each $i$, 
\[
f^\ast (T_i) = f^\ast (\omega_i) \wedge d^{-1} f^\ast (\omega_{i,1}) \wedge \ldots \wedge d^{-1}f^\ast (\omega_{i,K_i})
\]
where $f^\ast(\omega_i)$ is an $M_{i0}$ form and $f^\ast (\omega_{i,j})$ is an $M_{i,j}$ form.
Then as above,
\[
\begin{split}
&\|\lapms{1-\alpha_1} f^\ast(T_i)\|_{L^{\frac{N}{M_i-(1-\alpha_1)}}(\Ep^{M_i} \R^N)}\\
   \aleq & \|\lapms{\alpha_i} f^\ast(\omega)\|_{L^{\frac{N}{M_{i0}-\alpha_i}}}\, \|\lapms{1-\alpha_i} f^\ast(\omega_{i1})\|_{L^{\frac{N}{M_{i1}-(1-\alpha_i)}}  }\, \ldots\, \|\lapms{1-\alpha_i} f^\ast(\omega_{iK_i})\|_{L^{\frac{N}{M_{iK_i}-(1-\alpha_i)}}  }\\
 \aleq& [f]_{W^{\beta,\frac{N}{\beta}}}^{\frac{M_i+K_i}{\beta}} + 1
\end{split}
\]
for $\alpha_i \leq 1-\alpha_1$ and $\beta>1-\frac{\alpha_i}{M_{i0}}, 1-\frac{1-\alpha_i}{M_{ij}}$ for all $j$.

In conclusion, we get
\[
\begin{split}
  &\| \lapms{\alpha_0}f^\ast(T)\|_{L^{\frac{N}{M_0-\alpha_0}}(\R^N)}\\
 \aleq & \brac{[f]_{W^{\beta,\frac{N}{\beta}}}^{\frac{M}{\beta}} + 1}\, \brac{[f]_{W^{\beta,\frac{N}{\beta}}}^{\frac{M_1+K_1}{\beta}} + 1}\, \ldots\,\brac{[f]_{W^{\beta,\frac{N}{\beta}}}^{\frac{M_L+K_L}{\beta}} + 1}\\
 \aleq & [f]_{W^{\beta,\frac{N}{\beta}}}^{\frac{M_0+L+K_1+\ldots + K_L}{\beta}} + 1
   \end{split}
\]
since $M_0 = M + (M_1-1) + (M_2-1) + \ldots + (M_L-1)$.
Note that $L+K_1+\ldots + K_L$ is the number of vertical arrows in $T$.

{\underline {General Case}}

Arguing by induction over the depth of the tree, denoting by $K_i$ the number of vertical arrows in $T_i$ we find
\[
\begin{split}
  &\| \lapms{\alpha_0}f^\ast(T)\|_{L^{\frac{N}{M_0-\alpha_0}}(\R^N)}\\
    \aleq & \|\lapms{\alpha_1} f^\ast(\omega)\|_{L^{\frac{N}{M-\alpha_1}}}\, \|\lapms{1-\alpha_1} f^\ast(\omega_1)\|_{L^{\frac{N}{M_1-(1-\alpha_1)}}(\Ep^{M_1} \R^N)}\, \ldots\, \|\lapms{1-\alpha_1} f^\ast(\omega_K)\|_{L^{\frac{N}{M_K-(1-\alpha_1)}}(\Ep^{M_K} \R^N)}\\
 %
%
\aleq & \brac{[f]_{W^{\beta,\frac{N}{\beta}}}^{\frac{M}{\beta}} + 1}\, \brac{[f]_{W^{\beta,\frac{N}{\beta}}}^{\frac{M_1+K_1}{\beta}} + 1}\, \ldots\,\brac{[f]_{W^{\beta,\frac{N}{\beta}}}^{\frac{M_L+K_L}{\beta}} + 1}\\
\aleq & [f]_{W^{\beta,\frac{N}{\beta}}}^{\frac{M_0+L+K_1+\ldots + K_L}{\beta}} + 1.
 \end{split}
\]

\end{proof}

Combining \Cref{la:lapmsest} and \Cref{la:lapmsest2} and \Cref{la:lapmsest3}, we obtain

%
%
%
\begin{proof}[Proof of Proposition~\ref{pr:mainharmest}]
The case where $T$ is a tree of depth $0$ can be found in \cite{SVS20}. So from now on we assume that $M_0 < N$ and $T$ has the structure
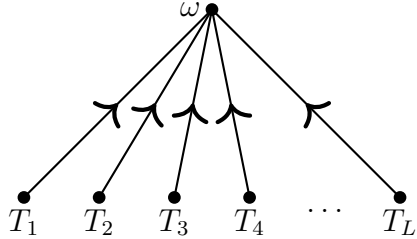
\begin{center}
\begin{tikzpicture}[scale = 1,thick, arr/.style={postaction={
      decorate,
      decoration={
    markings,
    mark=at position 0.5 with {\arrow[scale=2]{>}}}}}]
	\coordinate (omega_0) at (0.5,-0.5);
	\coordinate (T_1) at (-2,-3);
    \coordinate (T_2) at (-1,-3);
    \coordinate (T_3) at (0,-3);
    \coordinate (T_4) at (1,-3);
    \coordinate (T_5) at (2,-3);
     \coordinate (T_L) at (3,-3);

	\filldraw[black] (omega_0) circle (2pt) node[anchor=east]{$\omega$};
	\filldraw[black] (T_1) circle (2pt) node[anchor=north]{$T_1$};
	\filldraw[black] (T_2) circle (2pt) node[anchor=north]{$T_2$};
	\filldraw[black] (T_3) circle (2pt) node[anchor=north]{$T_3$};
	\filldraw[black] (T_4) circle (2pt) node[anchor=north]{$T_4$};
	\filldraw[black] (T_5) node[anchor=north]{$\ldots$};
	\filldraw[black] (T_L) circle (2pt) node[anchor=north]{$T_L$};
	\draw[arr] (T_1) -- (omega_0);
	\draw[arr] (T_2) -- (omega_0);
		\draw[arr] (T_3) -- (omega_0);
			\draw[arr] (T_4) -- (omega_0);
				\draw[arr] (T_L) -- (omega_0);
		
\end{tikzpicture}
\end{center}

Let 
\[
 p_0 = \frac{N}{M_0-\alpha_0}.
\]
By localization we need to find an estimate in $\R^N$
\[
\begin{split}
 &\int_{\R^N} f^\ast(\omega_0) \wedge d^{-1} f^\ast(T_1) \wedge \ldots \wedge d^{-1} f^\ast(T_L)\\
 =&\int_{\R^N} \lapms{\alpha_0}f^\ast(\omega_0) \wedge \Ds{\alpha_0}\brac{d^{-1} f^\ast(T_1) \wedge \ldots \wedge d^{-1} f^\ast(T_L)}\\
 \aleq&\| \lapms{\alpha_0}f^\ast(\omega_0)\|_{L^{p_0}(\R^N)}\, \|\Ds{\alpha_0}\brac{d^{-1} f^\ast(T_1) \wedge \ldots \wedge d^{-1} f^\ast(T_L)}\|_{L^{p_0'}(\R^N)}
 \end{split}
\]
In the second equality we used an integration by parts.
%
%
As in the proof of \Cref{la:lapmsest2}, by fractional Leibniz Rule, see e.g. \cite{LS20}, (observe $\frac{M_1-(1-\alpha_0)}{N} + \frac{M_2-1}{N} + \ldots + \frac{M_L-1}{N}=\frac{N-M_0+\alpha_0}{N}=\frac{1}{p_0'}$)
\[
\begin{split}
 &\|\Ds{\alpha_0}\brac{d^{-1} f^\ast(T_1) \wedge \ldots \wedge d^{-1} f^\ast(T_L)}\|_{L^{p_0'}(\R^N)}\\
 \aleq& \|\lapms{1-\alpha_0} f^\ast(T_1)\|_{L^{\frac{N}{M_1-(1-\alpha_0)}}(\R^N)} \cdot \ldots \cdot \|\lapms{1} f^\ast(T_L)\|_{L^{\frac{N}{M_L-1}}(\R^N)}\\
 &+\|\lapms{1} f^\ast(T_1)\|_{L^{\frac{N}{M_1-1}}(\R^N)} \cdot \|\lapms{1-\alpha_0} f^\ast(T_2)\|_{L^{\frac{N}{M_2-(1-\alpha_0)}}(\R^N)} \cdot \ldots \cdot \|\lapms{1} f^\ast(T_L)\|_{L^{\frac{N}{M_L-1}}(\R^N)}\\
 &+ \ldots + \|\lapms{1} f^\ast(T_1)\|_{L^{\frac{N}{M_1-1}}(\R^N)}  \cdot \ldots \cdot \|\lapms{1-\alpha_0} f^\ast(T_L)\|_{L^{\frac{N}{M_L-(1-\alpha_0)}}(\R^N)}.
 \end{split}
\]
Above we used again repeatedly the definition $d^{-1} f^\ast(\omega) =  d^\ast \lap^{-1} f^\ast(\omega)$. The claim now readily follows from \Cref{la:lapmsest} and \Cref{la:lapmsest3}.
\end{proof}
%
%
%

Let us mention that it seems to us that extending \Cref{la:lapmsest} to Triebel-Lizorkin spaces one can obtain the limit case $\beta = \beta_0$ from the above argument. We do not pursue this aspect further, though. Also, let us remark again that for any specific tree $T$ simple combinatorics allows to explicitly compute the best $\beta_0$ that our method allows. This was done for the Hopf degree in \cite{SVS20}.

\section{Proof of Theorem~\ref{th:main}}\label{s:proofmain}

While quantitative topology in general is very difficult and only special cases are known, let us begin with recalling a very easy statement: small BMO-norm means a function is contractible to a point (and thus topologically trivial). Here we recall the definition
\begin{equation}\label{eq:smallbmo}
\begin{split}
 [f]_{{\rm BMO}(\S^N)} :=& \sup_{r > 0, x\in \S^N} \mvint_{B(x,r)\cap \S^N} |f-(f)_{B(x,r)\cap \S^N)}|\\
 \aeq& \sup_{r > 0, x\in \S^N} \mvint_{B(x,r)\cap \S^N} \mvint_{B(x,r)\cap \S^N} |f(\theta)-f(\sigma)| d\sigma d\theta.
\end{split}
 \end{equation}
The following is the the precise statement
\begin{lemma}\label{la:smallnohomo}
Let $\n$ be a smooth compact manifold without boundary embedded into $\R^M$ and $\d: \pi_{N}(\n) \to \R$ be a rational homotopy group. Then there exists $\eps = \eps(\d)$ such that whenever $f \in \lip(\S^N,\n)$ satisfies one of the following
\begin{itemize}
 \item If $[f]_{{\rm BMO}} < \eps$, or
 \item if $[f]_{W^{\beta,\frac{N}{\beta}}} < \eps$ for some $\beta > 0$, or
 \item if $[f]_{C^{\beta}} < \eps$ for $\beta > 0$
\end{itemize}
then ${\d}([f]) = 0$.
\end{lemma}
\begin{proof}
The proof is standard, but we repeat it for the convenience of the reader.  

It is easy to check that the second and third condition imply the small BMO-norm, so we focus on this one. Let
\[
 F \in \lip(\B^{N+1},\R^M)
\]
be the harmonic extension of $f$, namely via the Poisson formula
\[
 F(x) =c(N) \int_{\S^{N}} f(\theta) \frac{1-|x|^2}{|x-\theta|^{N+1}}\, d\theta, \quad |x| < 1.
\]
The BMO-condition shows that $F$ stays close to the manifold $\mathcal{N}$. To see this observe that for any $\sigma \in \S^N$,
\[
 \dist(F(x),\n) \leq |F(x)-f(\sigma)|.
\]
Multiplying this with the Poisson kernel and integrating in $\sigma$ we see that 
\[
\begin{split}
 \dist(F(x),\n) \aleq& \int_{\S^{N}} |F(x) -f(\sigma)| \frac{1-|x|^2}{|x-\sigma|^{N+1}}\, d\sigma\\
 \aleq& \int_{\S^N} \int_{\S^{N}} |f(\theta) -f(\sigma)| \frac{1-|x|^2}{|x-\theta|^{N+1}}\, \frac{1-|x|^2}{|x-\sigma|^{N+1}}\, d\theta\, d\sigma.
 \end{split}
\]
Set $r = 1-|x|$ and $X := \frac{x}{|x|}$ and set $A(X,r,k) := B(X,2^{k+1} r) \setminus B(X,2^k r)$ if $k > 0$ and $A(X,r,0) = B(X, r)$. Then we have 
\[
\begin{split}
 \dist(F(x),\n) \aleq& \sum_{k \in \N \cup \{0\}} \sum_{\ell \in \N \cup \{0\}} \int_{\S^N \cap A(X,r,k)} \int_{\S^{N} \cap A(X,r,\ell)} |f(\theta) -f(\sigma)| \frac{r}{(2^k r)^{N+1}}\, \frac{r}{(2^\ell r)^{N+1}}\, \, d\theta\, d\sigma\\
 \aeq& \sum_{k \in \N \cup \{0\}} \sum_{\ell \in \N \cup \{0\}} 2^{-\ell} 2^{-k} \mvint_{\S^N \cap A(X,r,k)} \mvint_{\S^{N} \cap A(X,r,\ell)} |f(\theta) -f(\sigma)| \, d\theta\, d\sigma\\
 \aleq&\sum_{k \in \N \cup \{0\}} \sum_{\ell \in \N \cup \{0\}} 2^{-\ell} 2^{-k} [f]_{{\rm BMO}} \aeq [f]_{{\rm BMO}}.
 \end{split}
\]
That is, we have shown 
\[
 \dist(F(x),\n) \aleq [f]_{{\rm BMO}} < \eps.
\]
Since $\n$ is compact without boundary, there exists a $\delta > 0$ and the nearest point projection from a tubular neighborhood of $\n$ into $\n$, $\pi_{\n}: B_\delta(\n) \to \n$, see e.g. \cite{S96}. So if $\eps$ is small enough we have that $F(x): \B^{N+1} \to B_\delta(\n)$, and thus $G(x) := \pi_{\n} \circ F$ is well defined and as smooth as $F$. Now 
\[
 H(\theta,t) := G(t\theta) : \S^{N-1} \times  [0,1] \to \n
\]
is a smooth homotopy. That is $f = \pi_{\n} \circ f = H(\cdot,1)$ is homotopic to the constant map $G(0) = H(\cdot,0)$. So $f: \S^N \to \n$ is topologically trivial and thus ${\d}([f])=0$.
\end{proof}

From \Cref{la:smallnohomo} we conclude that when proving \Cref{th:main} we only need to consider maps $\varphi$ with relatively large $W^{\beta,\frac{N}{\beta}}(\S^N)$-norm, since for small $W^{\beta,\frac{N}{\beta}}(\S^N)$ norm the rational homotopy group vanishes. We want to apply \Cref{pr:mainharmest}. The connection to the rational homotopy group comes from the following representation formula for simply connected smooth manifolds $\n$ which are compact and without boundary.
It is due to Sullivan \cite{S77}, Novikov \cite{Nov1,Nov2,Nov3}, but in this form it was developed by Hardt and Rivi\`ere \cite{HR08}. We refer to \cite{HR08} for the proof. See also \cite{R07}.
\begin{proposition}\label{pr:hardtriviere}
Let ${\d}: \pi_{N}(\n) \to \R$ be a rational homotopy group, then
\[
 {\d}([f]) =  \sum_{k=1}^{K} \int_{\S^N} f^\ast(T_k)
\]
where $T_k$ is a tree-graph with $L_k$ vertical arrows.
\end{proposition}

\begin{proof}[Proof of Theorem~\ref{th:main}]
In view of \Cref{la:smallnohomo} the claimed estimate \eqref{eq:main:estW} is trivially satisfied if $[f]_{W^{\beta,\frac{N}{\beta}}} < \eps$ if $\eps=\eps(\d)$ is chosen small enough. 
On the other hand, combining \Cref{pr:hardtriviere} and \Cref{pr:mainharmest} with $L = \max L_k$ we obtain

\[
  |{\d}([f])| \aleq [f]_{W^{\beta,\frac{N}{\beta}}(\S^N)}^{\frac{N+L}{\beta}} + 1 \aleq_{\eps} [f]_{W^{\beta,\frac{N}{\beta}}(\S^N)}^{\frac{N+L}{\beta}}
\]
whenever $[f]_{W^{\beta,\frac{N}{\beta}}(\S^N)}^{\frac{N+L}{\beta}} \geq \eps$. This proves \eqref{eq:main:estW}. 
\end{proof}

\section{Examples and comparison with previous estimates}\label{s:examples}
%
Hardt--Rivi\`{e}re provided several examples of the generalized degree maps \cite{HR08}. We collect here the corresponding degree estimates.

\begin{example}
Let $\n = \bbbc P^2$.
There are two generalized degree maps, ${\d}_{\alpha} : \lip(\S^2,\bbbc P^2)$ and ${\d}_{\beta} : \lip(\S^5,\bbbc P^2)$ given by
\[
{\d}_{\alpha}(f) = \int_{\S^2} f^\ast \omega \quad \text{ and } \quad {\d}_{\beta}(f) = \int_{\S^5} f^\ast \omega^2 \wedge d^{-1} f^\ast \omega
\]
where $\omega$ is the K{\"a}hler form on $\bbbc P^2$.
Then the estimate is, for ${\beta_1} \ge \frac{3}{4}$ and ${\beta_2} \ge \frac{7}{8}$,
\[
\abs{ {\d}_{\alpha}(f)} \aleq  [f]_{W^{{\beta_1},\frac{2}{{\beta_1}}}(\S^2)}^{\frac{2}{{\beta_1}
}} \quad \text{ and } \quad \abs{ {\d}_{\beta}(f)} \aleq [f]_{W^{{\beta_2},\frac{5}{{\beta_2}}}(\S^5)}^{\frac{6}{{\beta_2}}}.
\]
\end{example}

\begin{example}
Let $\n = \S^2 \times \S^2$.
There are four generalized degree maps, ${\d}_{\alpha_i} : \lip(\S^2,\S^2 \times \S^2)$ and ${\d}_{\beta_i} : \lip(\S^3,\S^2 \times \S^2)$ for $i=1,2$ given by
\[
{\d}_{\alpha_i}(f) = \int_{\S^2} f^\ast \omega_i \quad \text{ and } \quad {\d}_{\beta_i}(f) = \int_{\S^3} f^\ast \omega_i \wedge d^{-1} f^\ast \omega_i
\]
where $\omega_i$ is a pull-back of the generator of $H^2(\S^2)$ under the coordinate projections $\pi_i : \S^2 \times \S^2 \to \S^2$.
Then the estimate is, for ${\beta_1} \ge \frac{3}{4}$ and ${\beta_2} \ge \frac{3}{4}$,
\[
\abs{ {\d}_{\alpha_i}(f)} \aleq [f]_{W^{{\beta_1},\frac{2}{{\beta_1}}}(\S^2)}^{\frac{2}{{\beta_1}}} \quad \text{ and } \quad \abs{ {\d}_{\beta_i}(f)} \aleq [f]_{W^{{\beta_2},\frac{3}{{\beta_2}}}(\S^3)}^{\frac{4}{{\beta_2}}}.
\]
\end{example}

\begin{example}
Let $\n = (\S^2 \times \S^2) \# \bbbc P^2$.
There are many generalized degree maps, but here we only consider maps from $\S^4$.
There are five of them, namely, ${\d}_{\gamma_i}, {\d}_{\delta_{k}} : \lip(\S^4,(\S^2 \times \S^2) \# \bbbc P^2)$ for $i=1,2,3$ and $k = 1,2$ given by
\[
\begin{split}
{\d}_{\gamma_i}(f) &= \int_{\S^4} f^\ast \omega_{1_i} \wedge d^{-1} f^\ast \omega_{2_i} \wedge d^{-1} f^\ast \omega_{3_i}\\
{\d}_{\delta_k}(f) &= \int_{\S^4} f^\ast \eta \wedge d^{-1} f^\ast \omega_{0_k} + \sum_{i=1}^{3} \int_{\S^4} f^\ast \omega_{1_{k,i}} \wedge d^{-1} f^\ast \omega_{2_{k,i}} \wedge d^{-1} f^\ast \omega_{3_{k,i}}
\end{split}
\]
where $\omega_i$, $i=1,2,3$ is a generator of $H^2(\n)$ and $\eta$ is a 3-form on $\n$ obtained from $\omega_i$.
Then the estimate is, for $\beta \ge \frac{3}{4}$,
\[
\abs{ {\d}_{\gamma_i}(f)} \aleq [f]_{W^{\beta,\frac{4}{\beta}}(\S^4)}^{\frac{6}{\beta}} \quad \text{ and } \quad \abs{ {\d}_{\delta_k}(f)} \aleq [f]_{W^{\beta,\frac{4}{\beta}}(\S^4)}^{\frac{6}{\beta}}.
\]
\end{example}

\begin{example}
In \cite{SVS20} they obtained for maps $f: \S^{4n-1} \to \S^{2n}$,
\[
\abs{ {\d}_{H}(f)} \aleq [f]_{W^{\beta,\frac{4n-1}{\beta}}(\S^{4n-1})}^{\frac{4n}{\beta}}.
\]
for $\beta \ge \frac{4n-1}{4n}$. See also \cite{Riv98} where Rivi\`{e}re proved the case $\beta=1$.
Note that
\[
{\d}_{H}(f) = \int_{\S^{4n-1}} \eta \wedge d\eta =  \int_{\S^{4n-1}} f^* \omega \wedge d^{-1} f^* \omega
\]
where $d\eta = f^* \omega$, $\eta$ is a smooth $2n-1$ form on $\S^{4n-1}$, $\omega$ is the volume form of $\S^{2n}$.
So, our estimate gives the same result because
\[
\abs{ {\d}_{H}(f)} \aleq [f]_{W^{\beta,\frac{N}{\beta}}(\S^N)}^{\frac{N+L}{\beta}} = [f]_{W^{\beta,\frac{4n-1}{\beta}}(\S^{4n-1})}^{\frac{4n}{\beta}}
\]
for $\beta \geq \beta_0 = 1-\frac{1}{2(2n)} = \frac{4n-1}{4n}$.
\end{example}

\appendix
\section{From Lipschitz to H\"older}
For maps between manifolds, a Lipschitz estimate on a homotopy invariant readily implies a H\"older estimate using a mollification argument. We recall this well-known fact here.

\begin{lemma}\label{la:LiptoCalpha}
Let $\m \subset \R^M$, $\n \subset \R^N$ be two smooth compact manifolds without boundary and assume that there exist a map 
\[
 \d:  C^0(\m,\n) \to \R 
\]
such that
\begin{itemize}
 \item $\d(f) = \d(g)$ if $f, g \in C^0(\m,\n)$ are homotopic to each other, and
 \item there are $\Lambda, q > 0$ we have
\[
 |\d(f)|\leq \Lambda [f]_{\lip}^q \quad \forall f \in \lip(\m,\n).
\]
\end{itemize}
Then for any $\alpha \in (0,1)$ there exists $C = C(\Lambda,q,\m,\n)$ such that 
\[
 |\d(f)|\leq C [f]_{C^{\alpha}}^{\frac{q}{\alpha}} \quad \forall f \in C^{\alpha}(\m,\n).
\]
\end{lemma}
\begin{proof}
Since $\n \subset \R^N$ is smooth and compact, there exists some $\sigma = \sigma(\n) > 0$ and a smooth nearest point projection from a tubular neighborhood of $\n$ into $\n$, \[\pi: B_\sigma(\n) \to \n.\]

Combining a decomposition of unity on $\m$ with the usual mollification and the projection $\pi$ we find that for any $\eps > 0$ and any $f \in C^\alpha(\m,\n)$ there exists a map $f_\eps$ such that 
\[
 \|f_\eps - f\|_{L^\infty(\m)} \leq \eps^{\alpha} [f]_{C^\alpha}.
\]
and with some uniform constant $C_1(\m,\n) > 0$
\[
 [f_\eps]_{\lip} \leq C_1(\m,\n)\, \eps^{\alpha-1} [f]_{C^\alpha}.
\]
Set for $\delta > 0$
\[
 \eps := \delta [f]_{C^\alpha}^{-\frac{1}{\alpha}}.
\]
then we have 
\[
 \|f_\eps - f\|_{L^\infty(\m)} \leq \delta^{\alpha}.
\]
and
\[
 [f_\eps]_{\lip} \leq C_1(\m,\n)\, \delta^{\alpha-1} [f]_{C^\alpha}^{\frac{1}{\alpha}}.
\]
Take $\delta := \frac{1}{2} \sigma^{\frac{1}{\alpha}}$. Then
\[
 F(t,\cdot) := \pi_{\n}\brac{(1-t) f_\eps + t f}: \quad \m \to \n
\]
is well defined for all $t \in [0,1]$, and thus $f_\eps$ and $f$ are homotopic to each other.

We conclude 
\[
 |\d(f)| = |\d(f_\eps)| \leq \Lambda [f_\eps]_{\lip}^q = \Lambda C_1(\m,\n)^q\, \delta^{(\alpha-1)q}\, [f]_{C^\alpha}^{\frac{q}{\alpha}}.
\]
Setting $C := \Lambda C_1(\m,\n)^q\, \delta^{(\alpha-1)q}$ we can conclude.

\end{proof}

\bibliographystyle{abbrv}
\bibliography{bib}

\end{document}